\documentclass[10pt,a4paper]{article}

\usepackage[latin2]{inputenc}

\usepackage[T1]{fontenc}
\usepackage{amsmath}
\usepackage{amssymb}
\usepackage{amsthm}
\usepackage{bbm}

\newtheorem{tw}{Theorem}
\newtheorem{lm}{Lemma}
\newtheorem{pr}{Proposition}
\newtheorem{wn}{Corollary}
\theoremstyle{definition}
\newtheorem{df}{Definition}
\newtheorem{uw}{Remark}

\newcommand{\R}{\mathbb{R}}
\newcommand{\N}{\mathbb{N}}
\newcommand{\Z}{\mathbb{Z}}
\newcommand{\cT}{\mathcal{T}}
\newcommand{\cS}{\mathcal{S}}
\newcommand{\cB}{\mathcal{B}}
\newcommand{\cC}{\mathcal{C}}
\newcommand{\raz}{\mathbbm{1}}

\begin{document}
%\linespread{2}
\bibliographystyle{plain}

\title{A note on the isomorphism of Cartesian products of ergodic flows}

\author{Joanna Ku\l{}aga\\{\small Faculty of Mathematics and Computer Science},\\{\small Nicolaus Copernicus University},\\{\small ul. Chopina 12/18, 87-100 Toru\'n, Poland}\\ {\small e-mail: joanna.kulaga@gmail.com}}%\\ {\small phone number: 0048 56 6113462} }
\date{}

\maketitle
%\begin{running}
%A note on the isomorphism of Cartesian products of ergodic flows
%\end{running}

\makeatletter{\renewcommand*{\@makefnmark}{}
\makeatother}

%\newpage
\abstract{ We show an isomorphism stability property for Cartesian products of either flows with joining primeness property or flows which are $\alpha$-weakly mixing.}


\thispagestyle{empty}

\section{Introduction}

The isomorphism problems in ergodic theory cover a broad spectrum of issues, e.g. the question when particular dynamical systems are metrically or spectrally isomorphic or the task of a more general nature to find invariants which distinguish dynamical systems. Our motivation in this note is a question posed by J.-P.~Thouvenot, which lies slightly outside the scope of the above-mentioned problems
\begin{equation}\label{Q:jpt} 
\begin{array}{l} 
\mbox{whether an isomorphism of Cartesian squares of $T$ and $S$}\\
\mbox{respectively, implies an isomorphism of $T$ and $S$.}
\end{array}
\end{equation}

Thouvenot's question still remains open in the class of all automorphisms. The first step toward a solution was made by V.V.~Ryzhikov in~\cite{ryz}. He proved that automorphisms $T$ and $S$ are isomorphic, provided that their Cartesian powers $T^{\times d}$ and $S^{\times d}$ (for some $d\geq 1$) are isomorphic and $T$ is $\alpha$-weakly mixing. This result, taking into account that $\alpha$-weak mixing is generic~\cite{St}, gives therefore the positive answer to~\eqref{Q:jpt} for a typical automorphism. In~\cite{ryz-tro} V.V.~Ryzhikov and A.E.~Troitskaya strengthened the result from~\cite{ryz} by replacing $\alpha$-weak mixing with the existence of a polynomial in the weak closure of time automorphisms:
\begin{tw}[\cite{ryz-tro}]\label{tw1} 
Let $T$ and $S$ be ergodic automorphisms of probability standard Borel spaces. Assume that for some $n_k\to\infty$
\begin{equation}\label{eq:gwiazdka}
T^{n_k}\to a\cdot\Pi+(1-a)\sum_{i\in\Z}a_iT^i
\end{equation}
for $a_i\geq 0$ for $i\in\Z$ and $a\geq 0$ such that $a+\sum_{i\in\Z}a_i=1$ with at least two summands positive.\footnote{The convergence in~\eqref{eq:gwiazdka} is the weak convergence of Markov operators of the $L^2(X,\mathcal{B},\mu)$-space on which the automorphism $T$ acts by $Tf=f\circ T$ and $\Pi f=\Pi_{X,X}f:= \int f\ d\mu$.}

If, for some $d\geq1$, $T^{\times d}$ is isomorphic to $S^{\times d}$ then $T$ is isomorphic to $S$.
\end{tw}
In~\cite{Tro} A.E.~Troitskaya extended the above result to the case of $\mathbb{Z}^2$-actions.

It was known that the answer to question~\eqref{Q:jpt} is positive for simple systems as the structure of factors of their Cartesian products had been described thoroughly~\cite{dJ-R}. In general however the structure of factors of Cartesian products for a given system may be very complex which explains the difficulties in providing a full answer to Thouvenot's question.

Theorem~\ref{tw1} gives a result for automorphisms. In a letter, V.~V.~Ryzhikov asked whether there is its natural counterpart for flows (i.e. we replace in~\eqref{eq:gwiazdka} the sum $\sum_{i\in\mathbb{Z}}a_i T^i$ with $\int T_t \ dP(t)$ for a Borel probability measure $P$ on $\mathbb{R}$). One of our aims is to give a partial positive answer to this question. In fact, we will deal with a more general problem, see~(\ref{pr1}) below.

It is not hard to see that the assumptions of Theorem~\ref{tw1} force $T$ to be weakly mixing, hence $S$ is also weakly mixing. Here, we will deal with weakly mixing flows. The problem we intend to consider is the following more general form of~\eqref{Q:jpt}. Assume that ${\cal T}_1,\ldots,{\cal T}_d$ are weakly mixing flows.
\begin{equation}\label{pr1} 
\begin{array}{l} 
\mbox{Suppose that ${\cal T}_1\times\ldots\times {\cal T}_d$ is isomorphic}\\ 
\mbox{to a product flow ${\cal S}_1\times\ldots\times{\cal S}_d$.}\\ 
\mbox{Is it true that there is a permutation $\sigma$ of $\{1,\ldots,d\}$ such that}\\
\mbox{${\cal T}_i$ is isomorphic to ${\cal S}_{\sigma(i)}$ for $i=1,\ldots,d$?}
\end{array}
\end{equation}
\begin{uw}
Note that while the answer to (1) remains unclear (and it is still plausible that it is positive), in general, the answer to (3) is negative. Indeed, assume that $\cT_1,\cT_2,,\cT'_1,\cT'_2$ are Bernoulli flows such  that
$$
h(\cT_1)+h(\cT_2)=h(\cT'_1)+h(\cT'_2)$$
but $h(\cT_1)\neq h(\cT'_1)\neq h(\cT_2)$.  In view of Ornstein's theorem~\cite{MR981173} $\cT_1\times \cT_2$ is isomorphic
 to $\cT_1'\times \cT_2'$, while $\cT'_1$ is neither isomorphic to $\cT_1$ nor to $\cT_2$.  Another example can be found in the class of Gaussian systems with simple spectrum. Indeed, if $\sigma,\sigma_1,\sigma_2,\sigma'_1,\sigma'_2$ are finite positive Borel measures on $\mathbb{R}$ such that
\begin{align*}
&\sigma_1\perp\sigma_2,\ \sigma_1'\perp\sigma'_2,\\
&\sigma=\sigma_1+\sigma_2=\sigma_1'+\sigma_2',\\
 &\sigma'_1\neq\sigma_1\neq\sigma'_2
\end{align*}
and the Gaussian flow $\cT_\sigma$ determined by $\sigma$ has simple spectrum then by the theory of Gaussian systems with ergodic self-joinings Gaussian \cite{joi2} we have
$$\cT_\sigma \simeq \cT_{\sigma_1}\times \cT_{\sigma_2}\simeq \cT_{\sigma'_1}\times \cT_{\sigma'_2}$$
 while $\cT_{\sigma_1}$ is neither isomorphic to $\cT_{\sigma'_1}$ nor to $\cT_{\sigma'_2}$.
\end{uw}

The main result of this note is the following theorem (see Section~\ref{se:deto} for needed definitions).
\begin{tw}\label{tw2}
The answer for~\eqref{pr1} is positive in the following cases:
\begin{itemize}
\item[(i)]
when $\cT_1,\dots,\cT_d$ are weakly mixing and satisfy the $JP$ property,
\item[(ii)]
when $\cT_i$ are  $\alpha_i$-weakly mixing for $\alpha_i\in (0,1)$ for $1\leq i \leq d$.
\end{itemize}
\end{tw}

Our main tool will be the theory of joinings. Apart from the JP property we will also study some joining properties of similar flavour for $\alpha$-weakly mixing flows. The generalization of part $(i)$ of Theorem~\ref{tw2} to the actions of other abelian Polish groups is straightforward, as the notion of JP is independent of the acting group.

As M.~Lema\'{n}czyk and V.V.~Ryzhikov noticed in private communication, the methods used in this note are not sufficient to answer question~\eqref{pr1} when we consider flows such that 
\begin{equation}\label{3jak2}
T^{t_n} \to a\cdot \Pi+ (1-a)\cdot \int T^t dP(t)
\end{equation}
for some $t_n\to\infty$ and $a\in [0,1)$ without any further assumption on measure $P$. Therefore the question whether Theorem~\ref{tw1} has a full counterpart for flows remains open. It also seems to be an open problem whether~\eqref{eq:gwiazdka} (or~\eqref{3jak2} in case of flows) implies the JP property or a weaker property in the spirit of Proposition~\ref{pr:15}. If $a_i$'s in~\eqref{eq:gwiazdka} decrease to zero exponentially fast then we obtain an analytic function in the weak closure of time automorphisms, which yields the CS property, hence the JP property (see~\cite{lpr}). The same mechanism works for flows and hence we have the following corollary of Theorem~\ref{tw2} (i).
\begin{wn}
The answer for~\eqref{pr1} is positive whenever~\eqref{3jak2} holds and $P$ is non-Dirac with Fourier transform $\hat{P}$ analytic. In particular, the answer for~\eqref{pr1} is positive if $P$ is continuous and has a bounded support.
\end{wn}
This gives a positive partial answer to the original question by V.~V. Ryzhikov.

In fact in the process of proving Theorem~\ref{tw2} we show more: the obtained isomorphisms between $\mathcal{T}_i$'s and $\mathcal{S}_i$'s are restrictions of the original isomorphism between the Cartesian products. In particular, this implies that the centralizer of the product $\mathcal{T}_1\times \dots \times \mathcal{T}_d$ is the product of the centralizers of $\mathcal{T}_1, \dots , \mathcal{T}_d$ up to a permutation of the coordinates, see Corollary~\ref{co:ce1} and Corollary~\ref{co:ce2}.

Similar problems to what we consider here were taken up in~\cite{dJ-L}. It was shown that for a typical automorphism and any $k_1,\dots, k_d, k'_1,\dots, k'_d\in \mathbb{N}$ the convolutions $\sigma_{T^{k_1}}\ast \ldots \ast \sigma_{T^{k_d}}$ and $\sigma_{T^{k'_1}}\ast \ldots \ast \sigma_{T^{k'_d}}$ are mutually singular provided that $(k_1,\dots, k_d)$ is not a rearrangement of $(k'_1,\dots, k'_d)$. This property (which can be viewed as a variation of the CS property) has the following consequence: for a typical automorphism $T$, the only way that $T^l$ (for any $l\in \mathbb{Z}\setminus \{0\}$) can sit as a factor of $T^{k_1}\times \dots \times T^{k_n}\times \dots$ is inside the $i$-th coordinate $\sigma$-algebra for some $i$ with $k_i=l$. In particular (for a generic transformation) $C(T^{k_1}\times \dots \times T^{k_d})=\bigcup_{\pi}C(T^{k_{\pi(1)}})\times\dots\times C(T^{k_{\pi(d)}})$ where $\pi$ runs over the set of permutations of $\{1,\dots,d\}$ such that $\pi(i)=j$ implies $k_i=k_j$.
%%%%%%%%%%%%%%%%%%%%%%%%%%%%%%%%%%%%%%%%%%%%%%%%%%%%%%%%%%%%
\section{Definitions and tools}\label{se:deto}
%%%%%%%%%%%%%%%%%%%%%%%%%%%%%%%%%%%%%%%%%%%%%%%%%%%%%%%%%%%%
\subsection{Joinings}
%%%%%%%%%%%%%%%%%%%%%%%%%%%%%%%%%%%%%%%%%%%%%%%%%%%%%%%%%%%%
Let us recall now the necessary information about joinings.\footnote{For more information on the theory of joinings we refer the reader e.g. to~\cite{joi0},~\cite{joi2} or~\cite{joi1}.} Let $\mathcal{T}=(T^t)_{t\in\mathbb{R}}$ and $\mathcal{S}=(S^t)_{t\in\mathbb{R}}$ be measurable flows on $(X,\mathcal{B},\mu)$ and $(Y,\mathcal{C},\nu)$ respectively (by measurability of the flow $(T^t)_{t\in\mathbb{R}}$ we mean that the map $\mathbb{R} \ni t \mapsto \left\langle f\circ T^t,g \right\rangle\in\mathbb{C}$ is continuous for all $f,g \in L^2 (X,\mathcal{B},\mu)$). By $J(\mathcal{T},\mathcal{S})$ we denote the set of all joinings between $\mathcal{T}$ and $\mathcal{S}$, i.e. the set of all $(T^t\times S^t)_{t\in\mathbb{R}}$-invariant probability measures on $(X\times Y,\mathcal{B}\otimes \mathcal{C})$, whose projections on $X$ and $Y$ are equal to $\mu$ and $\nu$ respectively. For $J(\mathcal{T},\mathcal{T})$ we write $J(\mathcal{T})$ and we denote the subspace of ergodic joinings by adding the superscript $e$: $J^e(\cT,\cS)$. Joinings are in one-to-one correspondence with Markov operators $\Phi\colon L^2(X,\mathcal{B},\mu)\to L^2(Y,\mathcal{C},\nu)$ satisfying $\Phi\circ T^t=S^t \circ \Phi$ for all $t\in\mathbb{R}$: 
\begin{align*}
&\Phi \mapsto \lambda_{\Phi} \in J(\cS,\cT),\ \lambda_{\Phi}(A\times B)= \int_B \Phi(\raz_A)\ d\nu,\\
&\lambda \mapsto \Phi_{\lambda}, \int \Phi_{\lambda}(f)(y)g(y)\ d\nu(y)=\int f(x)g(y)\ d\lambda(x,y).
\end{align*}
We denote by $\Pi_{X,Y}$ the Markov operator corresponding to the product measure $\mu\otimes \nu$. We denote the set of intertwining Markov operators also by $J(\mathcal{T},\mathcal{S})$. This identification allows us to view $J(\mathcal{T})$ endowed with the weak operator topology as a metrisable compact semitopological semigroup. 

It is said that $\mathcal{T}$ and $\mathcal{S}$ are \emph{disjoint}~\cite{Furstenberg67} if $J(\mathcal{T},\mathcal{S})=\{\mu\otimes\nu\}$; we write $\cT\perp \cS$. Given a flow $\mathcal{T}=(T^t)_{t\in\mathbb{R}}$ and a Borel probability measure $P$ on $\mathbb{R}$, we define the Markov operator $\int_{\mathbb{R}}T^t dP(t)$ acting on $L^2(X,\mathcal{B},\mu)$ by $\left\langle(\int_{\mathbb{R}}T^t dP(t))f,g\right\rangle=\int_{\mathbb{R}}\left\langle T^t f,g \right\rangle dP(t)$ for all $f,g\in L^2 (X,\mathcal{B},\mu)$.
%%%%%%%%%%%%%%%%%%%%%%%%%%%%%%%%%%%%%%%%%%%%%%%%%%%%%%%%%%%%
\subsection{JP property}
%%%%%%%%%%%%%%%%%%%%%%%%%%%%%%%%%%%%%%%%%%%%%%%%%%%%%%%%%%%%
We recall the notion of the joining primeness property (JP).
%%%%%%%%%%%%%%%%%%%%%%%%%%%%%%%%%%%%%%%%%%%%%%%%%%%%%%%%%%%%
\begin{df}[\cite{lpr}]
An ergodic flow $\mathcal{T}$ is said to have the \begin{em}joining primeness property\end{em}\footnote{The property~\eqref{eq:toto} defining the JP notion has been observed earlier for the class of quasi-simple systems in~\cite{Ry-Th}.} (JP) if for any weakly mixing flows $\cS_1,\cS_2$ for every $\lambda\in J^e(\mathcal{T},\cS_1\times \cS_2)$ we have
\begin{equation}\label{eq:toto}
\lambda=\lambda_{X, Y_1}\otimes \nu_2 \text{ or }\lambda=\lambda_{X,Y_2}\otimes \nu_1,
\end{equation} 
where $\lambda_{X,Y_1}$ and $\lambda_{X,Y_2}$ stand for the projections of $\lambda$ onto the appropriate coordinates.
\end{df}
%%%%%%%%%%%%%%%%%%%%%%%%%%%%%%%%%%%%%%%%%%%%%%%%%%%%%%%%%%%%
\begin{uw}\label{uw:uwaga}
In terms of Markov operators, the JP property means that for every $\Phi \in J^e(\mathcal{T},\cS_1\times \cS_2)$
\begin{equation*}
\Phi=p_1\circ \Phi\text{ or }\Phi=p_2\circ \Phi,
\end{equation*}
where $p_1\colon L^2(Y_1\times Y_2)\to L^2(Y_1)\otimes \raz$, $p_2\colon L^2(Y_1\times Y_2)\to \raz\otimes L^2(Y_2)$ are the orthogonal projections.
\end{uw}
%%%%%%%%%%%%%%%%%%%%%%%%%%%%%%%%%%%%%%%%%%%%%%%%%%%%%%%%%%%%
Let us now recall some properties of the class of flows enjoying the JP property.
\begin{pr}[\cite{lpr}]\label{pr:4}
The class of weakly mixing JP flows is closed under distal extensions which are weakly mixing.
\end{pr}
%%%%%%%%%%%%%%%%%%%%%%%%%%%%%%%%%%%%%%%%%%%%%%%%%%%%%%%%%%%%
The next result is a spectral criterion for the JP property.
%%%%%%%%%%%%%%%%%%%%%%%%%%%%%%%%%%%%%%%%%%%%%%%%%%%%%%%%%%%%
\begin{pr}[\cite{lpr}]\label{pr:my}
All weakly mixing flows whose maximal spectral type is singular with respect to the convolution of any two continuous measures (this is called in~\cite{lpr} the convolution simplicity property, i.e.\ the CS property) enjoy the JP property.
\end{pr}
%%%%%%%%%%%%%%%%%%%%%%%%%%%%%%%%%%%%%%%%%%%%%%%%%%%%%%%%%%%%
Recall that the CS property is generic in the class of flows on a fixed probability Borel space~\cite{lpr}. A stronger property than CS is the \emph{simple convolution} property (SC) which, by definition, holds when the Gaussian action determined by the reduced maximal spectral type $\sigma_{\mathcal{T}}$ of the given system has simple spectrum. In a recent paper~\cite{lemanczyk+parreau} it has been shown that there are natural classes of flows with the SC property: a typical flow on a fixed probability Borel space, special flows over a rotation by a ``generic'' $\alpha\in[0,1)$ under a smooth roof function which is not a trigonometrical polynomial and special flows over rotations by $\alpha\in [0,1)$ with unbounded partial quotients under some piecewise absolutely continuous roof function.

Notice that by Proposition~\ref{pr:4} and Proposition~\ref{pr:my} whenever a weakly mixing flow enjoys the CS property then its weakly mixing distal extension has the JP property.

A natural source of examples of flows with the CS property is given by the following result.
\begin{pr}[\cite{lpr}]\label{pr:wy}
Assume that for $t_n\to\infty$
\begin{equation}\label{eq:zbieznosc}
T^{t_n} \to \int T^t dP(t)
\end{equation}
for a probability Borel measure $P$ on $\R$ which is not a Dirac measure and such that $\widehat{P}$ is analytic\footnote{ This, for example, holds whenever $P$ has bounded support.}. Then $\sigma=\sigma_{\cT}$ is singular with respect to the convolution of any two continuous measures.
\end{pr}
%%%%%%%%%%%%%%%%%%%%%%%%%%%%%%%%%%%%%%%%%%%%%%%%%%%%%%%%%%%%
Such flows can be obtained in a natural way, namely as smooth flows on orientable surfaces. In~\cite{Koc76} A.~V.~Kochergin proved that there is a natural class of flows on surfaces which are weakly mixing but not mixing. This result was later extended in~\cite{FLold} to a larger class of flows on surfaces. A property which (together with other properties of the flows under consideration) was used to prove the absence of mixing turned out to be one of the sufficient conditions to obtain in the weak closure of time automorphisms an operator of the form~\eqref{eq:zbieznosc} satisfying the assumptions of Proposition~\ref{pr:wy} (see~\cite{FL?}).
%%%%%%%%%%%%%%%%%%%%%%%%%%%%%%%%%%%%%%%%%%%%%%%%%%%%%%%%%%%%

Notice also that by repeating word for word the proof of Proposition~\ref{pr:wy} (see~\cite{lpr}) we obtain the same result in the situation where~\eqref{eq:zbieznosc} is replaced by
\begin{equation*}
T^{t_n} \to a\cdot \Pi+ (1-a)\cdot \int T^t dP(t)
\end{equation*}
for some $a\in(0,1)$.
%%%%%%%%%%%%%%%%%%%%%%%%%%%%%%%%%%%%%%%%%%%%%%%%%%%%%%%%%%%%
\subsection{Partial mixing, partial rigidity and $\alpha$-weak mixing}
%%%%%%%%%%%%%%%%%%%%%%%%%%%%%%%%%%%%%%%%%%%%%%%%%%%%%%%%%%%%
Let us recall some basic definitions which will be used in what follows.
\begin{df}
$\cT$ is \emph{$\alpha$-partially mixing} for $\alpha \in (0,1)$ along $t_n\to \infty$ if
\begin{equation*}
T^{t_n} \to \alpha \cdot \Pi + (1-\alpha)\cdot J
\end{equation*}
for some $J \in J(\cT)$.
\end{df}
\begin{df}
$\cT$ is \emph{$(1-\alpha)$-partially rigid} for $\alpha \in (0,1)$ along $t_n \to \infty$ if
\begin{equation*}
T^{t_n} \to \alpha \cdot J+ (1-\alpha) \cdot Id
\end{equation*}
for some $J \in J(\cT)$.
\end{df}
\begin{df}[\cite{Ka01},\cite{St}]
$\cT$ is \emph{$\alpha$-weakly mixing} for $\alpha \in (0,1)$ along $t_n \to \infty$ if
\begin{equation*}
T^{t_n} \to \alpha \cdot \Pi + (1-\alpha)\cdot Id,
\end{equation*}
i.e. it is $\alpha$-partially mixing and $(1-\alpha)$-partially rigid along the same subsequence.
\end{df}
%%%%%%%%%%%%%%%%%%%%%%%%%%%%%%%%%%%%%%%%%%%%%%%%%%%%%%%%%%%%
\section{Results}
%%%%%%%%%%%%%%%%%%%%%%%%%%%%%%%%%%%%%%%%%%%%%%%%%%%%%%%%%%%%
\subsection{Isomorphism problem and JP property}
%%%%%%%%%%%%%%%%%%%%%%%%%%%%%%%%%%%%%%%%%%%%%%%%%%%%%%%%%%%%
The following is an immediate consequence of the definition of the JP property.
\begin{pr}\label{pr:piec}
Let $\cT$ enjoy the JP property and let $\cS_1,\cS_2$ be weakly mixing. Assume that $\cT$ is a factor of $\cS_1\times\cS_2$. Then $\cT$ is a factor of $\cS_1$ or $\cS_2$.
\end{pr}
%%%%%%%%%%%%%%%%%%%%%%%%%%%%%%%%%%%%%%%%%%%%%%%%%%%%%%%%%%%%
\begin{proof}
Let $\Phi \colon Y_1\times Y_2 \to X$ be such that $\cT \circ \Phi = \Phi \circ (\cS_1\times\cS_2)$. Then
$\Phi \colon L^2(X) \to L^2(Y_1\times Y_2) \text{ given by }\Phi f=f\circ\Phi$
determines an ergodic joining, whence by Remark~\ref{uw:uwaga}, we have $\Phi=p_1\circ \Phi$ or $\Phi=p_2\circ \Phi$ which means that
\begin{equation*}
\Phi(L^2(X)) \subset L^2(Y_1)\otimes \raz \text{ or }\Phi(L^2(X)) \subset \raz \otimes L^2(Y_2).
\end{equation*}
This completes the proof.
\end{proof}
%%%%%%%%%%%%%%%%%%%%%%%%%%%%%%%%%%%%%%%%%%%%%%%%%%%%%%%%%%%%
\begin{uw}\label{uw:inpart}
In particular, the proof of Proposition~\ref{pr:piec} shows the following: if $\mathcal{B}\subset \mathcal{C}_1\otimes \mathcal{C}_2$ is a factor-$\sigma$-algebra of $\cS_1\times \cS_2$ representing the action $\cT$ then either 
$\mathcal{B}\subset \mathcal{C}_1$ or $\mathcal{B}\subset \mathcal{C}_2$. The extension of this fact to more than two flows $\cS_i$ is straightforward (the proofs can be repeated word for word). More precisely, for any $d\geq 1$ and $\Phi_\lambda\in J^e(\cT,\cS_1\times\dots\times \cS_d)$ the inclusion
\begin{equation*}
\Phi_\lambda(L^2(X))\subset \sum_{i=1}^{d}L^2(Y_i)
\end{equation*}
implies that 
\begin{equation*}
\Phi_\lambda(L^2(X))\subset L^2(Y_i)\text{ for some }1\leq i\leq d.
\end{equation*}
\end{uw}
%%%%%%%%%%%%%%%%%%%%%%%%%%%%%%%%%%%%%%%%%%%%%%%%%%%%%%%%%%%%
\begin{proof}[Proof of Theorem~\ref{tw2}, part (i)]
We will provide the proof for $d=2$. The extension to the product of $d$ JP flows is straightforward and~\eqref{pr1} holds whenever $\cT_1,\dots,\cT_d$ are JP.

Let $\Phi \colon Y_1\times Y_2 \to X_1\times X_2$ determine the isomorphism between $\cS_1\times\cS_2$ and $\cT_1\times \cT_2$. Using Proposition~\ref{pr:piec}, we may assume without loss of generality that
\begin{equation*}
\Phi(L^2(X_1)\otimes \raz) \subset L^2(Y_1)\otimes \raz \text{ and }\Phi(\raz \otimes L^2(X_2)) \subset \raz \otimes L^2(Y_2).
\end{equation*}
Since $\Phi$ is an isomorphism and the image of $\sigma$-algebras $\cB_1\otimes \{\emptyset,X_2\}$ and $\{\emptyset, X_1\}\otimes \cB_2$ via $\Phi^{-1}$ generates $\cC_1\otimes \cC_2$, we have more:
\begin{equation*}
\Phi(L^2(X_1)\otimes \raz) = L^2(Y_1)\otimes \raz \text{ and }\Phi(\raz \otimes L^2(X_2)) = \raz \otimes L^2(Y_2),
\end{equation*}
which completes the proof.
\end{proof}
%%%%%%%%%%%%%%%%%%%%%%%%%%%%%%%%%%%%%%%%%%%%%%%%%%%%%%%%%%%%
Moreover, the proof shows the following.
\begin{wn}\label{co:ce1}
Let $\cT_1,\dots,\cT_d$ be weakly mixing JP flows. Then the centralizer $C(\cT_1\times \dots \times \cT_d)$ of $\cT_1\times \dots \times \cT_d$ consists of transformations belonging to
\begin{equation*}
C(\cT_{\sigma(1)})\times \dots \times C(\cT_{\sigma(d)}),
\end{equation*}
where the permutation $\sigma\in S(d)$ is such that $\sigma(i)=j$ implies $\cT_i\simeq \cT_j$.
\end{wn}
%%%%%%%%%%%%%%%%%%%%%%%%%%%%%%%%%%%%%%%%%%%%%%%%%%%%%%%%%%%%
Thouvenot's question has clearly a negative answer in the infinite case. Notice that for any partition of $\N$ into subsets $\N_1,\N_2,\dots$ we have
\begin{equation*}
\cT_1\times\cT_2\times\dots \simeq \left(\times_{i\in\N_1}\cT_i\right) \times \left(\times_{i\in\N_2}\cT_i\right)\times\dots,
\end{equation*}
so it suffices to consider e.g. $\N_1=\{1,2\}$, $\N_i=\{i+1\}$ for $i\geq 2$, a weakly mixing $\cT$ such that $\cT \not\simeq \cT\times\cT$ and take $\cT_i=\cT$ for all $i\in\N$. 

Therefore, the best infinite version of Theorem~\ref{tw2} we can hope for is that the isomorphism $\Phi$ of $\cT_1\times \cT_2\times\dots$ and $\cS_1\times \cS_2\times\dots$ implies that there exists a partition of $\N$ into subsets $\N_1,\N_2,\dots$ such that $\times_{j\in \N_i}\cT_j$ is isomorphic to $\cS_i$ via $\Phi$. It turns out that this holds true if we assume that the flows $\cT_i$ enjoy the JP property.
\begin{pr}\label{pr:inftyJP}
Let $\mathcal{T}_i$ be weakly mixing flows satisfying the JP property for $i\geq 1$. For flows $\cS_i$ for $i\geq 1$ such that $\mathcal{T}_1\times\mathcal{T}_2\times \dots $ is isomorphic to $\mathcal{S}_1\times \mathcal{S}_2\times \dots $ via $\Phi\colon Y_1\times Y_2 \times \dots \to X_1\times X_2 \times \dots$, there exists a partition of $\N$ into subsets $\N_1,\N_2,\dots$ such that $\Phi$ determines an isomorphism between $\times_{j\in \N_i}\cT_j$ and $\cS_i$ for $i\geq 1$. If we additionally assume that also $\mathcal{S}_i$ enjoy the JP property for $i\geq 1$, then there exists a permutation $\sigma\colon \N\to\N$ such that $\Phi$ determines an isomorphism between $\mathcal{T}_i$ and $\mathcal{S}_{\sigma(i)}$.
\end{pr}
\begin{proof}
Fix $i\in\mathbb{N}$. Let $\Phi^{-1}(\mathcal{B}_i)\subset \mathcal{C}_1\otimes \mathcal{C}_2\otimes \dots$ be the factor-$\sigma$-algebra of $\mathcal{S}_1\times \mathcal{S}_2\times \dots$ representing the action $\mathcal{T}_i$. We claim that there exists $k\geq 1$ such that $\Phi^{-1}(\mathcal{B}_i) \subset \mathcal{C}_1\otimes\dots \otimes \mathcal{C}_k$. Indeed, notice that for $k\geq 1$ we have
\begin{equation*}
\Phi^{-1}(\mathcal{B}_i)\subset \left(\mathcal{C}_1\otimes\dots \otimes \mathcal{C}_k\right)\otimes \left(\mathcal{C}_{k+1}\otimes \mathcal{C}_{k+2} \otimes \dots \right).
\end{equation*}
By the JP property, either $\Phi^{-1}(\mathcal{B}_i) \subset \mathcal{C}_1\otimes\dots \otimes \mathcal{C}_k$ or $\Phi^{-1}(\mathcal{B}_i)\subset \mathcal{C}_{k+1}\otimes \mathcal{C}_{k+2} \otimes \dots$ If for all $k\geq 1$ we have $\Phi^{-1}(\mathcal{B}_i)\subset \mathcal{C}_{k+1}\otimes \mathcal{C}_{k+2} \otimes \ldots $ then
\begin{equation*}
\Phi^{-1}(\mathcal{B}_i)\subset \bigcap_{k\geq 1} \mathcal{C}_{k+1}\otimes \mathcal{C}_{k+2} \otimes \ldots 
\end{equation*}
and by the Kolmogorov's zero-one law $\Phi^{-1}(\mathcal{B}_i)$ is trivial, which is impossible. Therefore, for some $k\geq 1$
\begin{equation*}
\Phi^{-1}(\mathcal{B}_i) \subset \mathcal{C}_1\otimes\dots \otimes \mathcal{C}_k.
\end{equation*}
Using again the JP property, we obtain, as in the finite case, that for some $\sigma(i)\geq 1$
\begin{equation}\label{eq:disj}
\Phi^{-1}(\mathcal{B}_i) \subset \mathcal{C}_{\sigma(i)}.
\end{equation}
Setting for $i\geq 1$
\begin{equation*}
\N_i:=\{j\in\N\colon \Phi^{-1}(\mathcal{B}_j)\subset \mathcal{C}_i\}
\end{equation*}
we obtain a partition of $\N$ (indeed, $\Phi$ is an isomorphism, whence $\cup_{i\in \N}\N_i=\N$ and the sets $\N_i$ are disjoint by~\eqref{eq:disj}). Moreover, $\sigma$-algebras $\Phi^{-1}(\mathcal{B}_j)$ for $j\in\N_i$ are independent and generate the whole $\sigma$-algebra $\mathcal{C}_i$, whence
\begin{equation*}
\mathcal{S}_i\simeq \times_{j\in\N_i}\cT_j,
\end{equation*}
which completes the proof of the first assertion.

If we additionally assume that the flows $\cS_i$ enjoy the JP property for $i\geq 1$, the sets $\N_i$ for $i\geq 1$ are singletons and therefore determine the permutation $\sigma\colon \N\to\N$ such that $\Phi$ is an isomorphism between $\cT_i$ and $\cS_{\sigma(i)}$.
\end{proof}
%%%%%%%%%%%%%%%%%%%%%%%%%%%%%%%%%%%%%%%%%%%%%%%%%%%%%%%%%%%%
As a direct consequence of the above proposition we obtain the following result.
%%%%%%%%%%%%%%%%%%%%%%%%%%%%%%%%%%%%%%%%%%%%%%%%%%%%%%%%%%%%
\begin{wn}
Let $\cT_i$ for $i\geq 1$ be weakly mixing flows enjoying the JP property. Then the centralizer $C(\cT_1\times\cT_2\times \dots)$ of $\cT_1\times\cT_2\times \dots$ consists of transformations belonging to
\begin{equation*}
C(\cT_{\sigma(1)})\times C(\cT_{\sigma(2)})\times \dots,
\end{equation*}
where the permutation $\sigma\colon \N\to\N$ is such that $\sigma(i)=j$ implies $\cT_i\simeq \cT_j$.
\end{wn}
%%%%%%%%%%%%%%%%%%%%%%%%%%%%%%%%%%%%%%%%%%%%%%%%%%%%%%%%%%%%
\subsection{JP property as a weaker version of disjointness}
%%%%%%%%%%%%%%%%%%%%%%%%%%%%%%%%%%%%%%%%%%%%%%%%%%%%%%%%%%%%
Let $\cT$ and $\cS$ be ergodic flows on $(X,\cB,\mu)$ and $(Y,\cC,\nu)$ respectively. We will now compare $\alpha$-partial rigidity and $\beta$-partial mixing. Let us first recall a lemma.
\begin{lm}[\cite{lpr}]\label{lm:showjp}
Assume that $\lambda \in J^e(\cT,\cS_1\times \cS_2)$ satisfies
\begin{equation*}
\Phi_{\lambda}(L^2_0(X,\mathcal{B},\mu)) \subset L^2_0(Y_1,\cC_1,\nu_1) \oplus L^2_0(Y_2,\cC_2,\nu_2).
\end{equation*}
Then $\lambda=\lambda_{X,Y_1}\otimes \nu_2$ or $\lambda=\lambda_{X,Y_2}\otimes \nu_1$.
\end{lm}
%%%%%%%%%%%%%%%%%%%%%%%%%%%%%%%%%%%%%%%%%%%%%%%%%%%%%%%%%%%%
\begin{pr}\label{pr:15}
Let $\alpha,\beta\in[0,1]$. Assume that $\cT$ is $\alpha$-partially rigid and $\cS$ is $\beta$-partially mixing along the same time-sequence, i.e. for some $t_n \to \infty$ we have
\begin{equation*}
T^{t_n}\to \alpha\cdot \text{Id}+(1-\alpha)\cdot J\text{ and }S^{t_n} \to \beta\cdot\Pi+ (1-\beta)\cdot K
\end{equation*}
for some $J\in J(\cT)$ and $K\in J(\cS)$. 
\begin{itemize}
\item[(i)]
If $\alpha=\beta=1$ then $\cT$ and $\cS$ are spectrally disjoint.
\item[(ii)]
If $\alpha+\beta>1$ then $\cT\perp\cS$.
\end{itemize}
Let $\alpha\in(0,1)$ and let $\cS_1,\cS_2$ be ergodic flows on $(Y_1,\mathcal{C}_1,\nu_1)$ and $(Y_2,\mathcal{C}_2,\nu_2)$ respectively.
\begin{itemize}
\item[(iii)]
If $\cT$ is $(1-\alpha)$-weakly mixing\footnote{ When $\alpha+\beta=1$, it is not true anymore that $\cS\perp \cT$. It suffices to consider $\alpha$-weakly mixing $\cT=\cS$ 
and take their diagonal joining: $\Delta\in J(\cT,\cS)$ given by $\Delta(A\times B)=\mu(A\cap B)$ for $A,B\in\cB$. However, as $(iii)$ shows, in this case (provided that $\cT$ is $(1-\alpha)$-weakly mixing) we observe some kind of JP property for $\cT$.} and $\cS_1,\cS_2$ are $(1-\alpha)$-partially mixing along $t_n$ then  for every $\lambda\in J^e(\mathcal{T},\cS_1\times \cS_2)$ we have
\begin{equation*}
\lambda=\lambda_{X, Y_1}\otimes \nu_2 \text{ or }\lambda=\lambda_{X,Y_2}\otimes \nu_1,
\end{equation*} 
\end{itemize}
\end{pr}
%%%%%%%%%%%%%%%%%%%%%%%%%%%%%%%%%%%%%%%%%%%%%%%%%%%%%%%%%%%%
\begin{proof}
$(i)$ is well-known. 

We will now show that $(ii)$ holds. Suppose that there exists $\Phi\in J^e(\cT,\cS)$ such that $\Pi\neq\Phi$. We have
\begin{equation*}
S^{t_n}\Phi \to (\beta\cdot \Pi_{Y, Y} + (1-\beta)\cdot K)\Phi = \beta\cdot \Pi_{X,Y} + (1-\alpha)\cdot K\Phi.
\end{equation*}
On the other hand
\begin{equation*}
S^{t_n}\Phi = \Phi T^{t_n} \to \Phi (\alpha \cdot Id + (1-\alpha)\cdot J)= \alpha \cdot \Phi + (1-\alpha)\cdot\Phi J.
\end{equation*}
Therefore
\begin{equation}\label{eq:uuu}
\beta\cdot \Pi + (1-\alpha)\cdot K\Phi=\alpha\cdot \Phi + (1-\alpha)\cdot\Phi J.
\end{equation}
We have $\Pi,\Phi \in J^e(\cS,\cT)$. Using~\eqref{eq:uuu} and by the uniqueness of the ergodic decomposition we conclude that in the ergodic decomposition of $(1-\alpha)\cdot\Phi J$ we will see $\beta\cdot \Pi$. Hence $1-\alpha \geq  \beta$, which yields a contradiction. Therefore $J^e(\cS,\cT)=J(\cS,\cT)=\{\Pi\}$.

We will show now that $(iii)$ holds. Take $\lambda\in J^e(\cT,\cS_1\times \cS_2)$. In view of Lemma~\ref{lm:showjp} it suffices to show that for $\Phi=\Phi_{\lambda}$
\begin{equation*}
\Phi(L^2_0(X))\subset (L^2_0(Y_1)\otimes \raz) \oplus (\raz \otimes L^2_0(Y_2)).
\end{equation*}
Take $f\in L^2_0(X)$. Since
\begin{equation*}
L^2_0(Y_1\times Y_1)=(L^2_0(Y_1)\otimes \raz ) \oplus (\raz \otimes L^2_0(Y_2)) \oplus (L^2_0(Y_1)\otimes L^2_0(Y_2)),
\end{equation*}
for some $f_1\in L^2_0(Y_1)$, $f_2 \in L^2_0(Y_2)$ and $f_3 \in L^2_0(Y_1)\otimes L^2_0(Y_2)$ we have
\begin{equation*}
\Phi f= f_1\otimes \raz + \raz \otimes f_2 + f_3.
\end{equation*}
Therefore
\begin{multline*}
\langle (S_1\times S_2)^{t_n}\Phi f, \Phi f \rangle= \langle S_1^{t_n}f_1,f_1 \rangle + \langle S_2^{t_n}f_2,f_2\rangle + \langle (S_1\times S_2)^{t_n}f_3,f_3\rangle\\
\to \alpha \langle K_1f_1,f_1\rangle+ \alpha \langle K_2f_2,f_2\rangle+ \alpha^2 \langle (K_1\otimes K_2)f_3,f_3\rangle.
\end{multline*}
Hence
\begin{equation}\label{eq:mu1}
| \lim_{t_n\to \infty}\langle (S_1\times S_2)^{t_n}\Phi f, \Phi f \rangle| \leq \alpha \|f_1\|^2+\alpha \|f_2\|^2 +\alpha^2 \| f_3\|^2.
\end{equation}
On the other hand
\begin{multline}\label{eq:mu2}
\langle (S_1\times S_2)^{t_n}\Phi f, \Phi f \rangle=\langle \Phi T^{t_n} f,\Phi f \rangle=\langle T^{t_n}f,\Phi^{\ast }\Phi f\rangle\\
\to \alpha \langle f, \Phi^{\ast}\Phi f\rangle=\alpha \|\Phi f\|^2.
\end{multline}
Since $\|\Phi f\|^2=\|f_1\|^2+\|f_2\|^2+\|f_3\|^2$, \eqref{eq:mu1} and~\eqref{eq:mu2} may hold only when $f_3=0$, which completes the proof.
\end{proof}
%%%%%%%%%%%%%%%%%%%%%%%%%%%%%%%%%%%%%%%%%%%%%%%%%%%%%%%%%%%%
\subsection{Isomorphism problem and  $\alpha$-weak mixing}
%%%%%%%%%%%%%%%%%%%%%%%%%%%%%%%%%%%%%%%%%%%%%%%%%%%%%%%%%%%%
Proposition~\ref{pr:gl} below will complete the proof of Theorem~\ref{tw2}.
%%%%%%%%%%%%%%%%%%%%%%%%%%%%%%%%%%%%%%%%%%%%%%%%%%%%%%%%%%%%
\begin{pr}\label{pr:gl}
Let $d\geq 1$ and let $\cT_i$ be $\alpha_i$-weakly mixing for some $\alpha_i\in(0,1)$ for $1\leq i\leq d$ along $t_n\to\infty$. Let $\cS_i$ be weakly mixing for $1\leq i \leq d$. If $\cT_1\times \dots \times \cT_d\simeq \cS_1\times \dots \times \cS_d$ then there exists a permutation $\sigma\in S(d)$ such that $\cT_i \simeq \cS_{\sigma(i)}$. More precisely, if $\Phi \colon Y_1\times \dots \times Y_d \to X_1\times \dots \times X_d$ determines the isomorphism between $\cS_1\times \dots \times \cS_d$ and $\cT_1\times \dots \times \cT_d$ then %$\Phi|_{L^2(\mathcal{C}_{\sigma(i)})}$ yields an isomorphism between $\mathcal{S}_{\sigma(i)}$ and $\mathcal{T}_i$.
$\Phi|_{L^2(\mathcal{B}_{i})}$ yields an isomorphism between $\mathcal{T}_i$ and $\mathcal{S}_{\sigma(i)}$.
\end{pr}
Before we prove the result stated above, we need some auxiliary lemmas.
%%%%%%%%%%%%%%%%%%%%%%%%%%%%%%%%%%%%%%%%%%%%%%%%%%%%%%%%%%%%
\begin{lm}\label{lm:15}
Let $R\in J(\cT)$ where $\cT\simeq \cS$, i.e. $\cS\circ Q = Q\circ \cT$ for some isomorphism $Q\colon X \to Y$. If 
\begin{equation*}
R=\int V\ d P(V)
\end{equation*}
corresponds to the ergodic decomposition of $\lambda_R$ then
\begin{equation*}
Q^{-1}RQ=\int Q^{-1}VQ\ dP(V)=\int V \ d (\Phi^{-1}\circ\cdot \circ\Phi)_{\ast}(P)(V)
\end{equation*}
corresponds to the ergodic decomposition of $\lambda_{Q^{-1}RQ}$.
\end{lm}
%%%%%%%%%%%%%%%%%%%%%%%%%%%%%%%%%%%%%%%%%%%%%%%%%%%%%%%%%%%%
\begin{proof}
Notice that $\lambda_{Q^{-1}VQ}= \lambda_V \circ (Q \times Q)^{-1}.$
Indeed, for $f,g\in L^2(Y)$ we have
\begin{multline*}
\int f \otimes g \ d \lambda_{Q^{-1}VQ}=\int Q^{-1}VQ f\cdot g \ d \nu = \int VQf\cdot Qg\ d\mu \\
= \int Qf \otimes Qg \ d\lambda_V=\int (Q\otimes Q) (f\otimes g)\ d\lambda_V = \int f\otimes g \ d\lambda_V \circ (Q\otimes Q)^{-1}.
\end{multline*}
It is clear that
\begin{equation*}
Q\times Q \colon X\times X \to Y\times Y
\end{equation*}
yields an affine isomorphism of the simplices of joinings.
\end{proof}
%%%%%%%%%%%%%%%%%%%%%%%%%%%%%%%%%%%%%%%%%%%%%%%%%%%%%%%%%%%%
\begin{lm}\label{lm:18}
When the assumptions of Proposition~\ref{pr:gl} are satisfied, $\cS_i$ are $\beta_i$-weakly mixing for some $\beta_i \in (0,1)$ and $1\leq i\leq d$.
\end{lm}
%%%%%%%%%%%%%%%%%%%%%%%%%%%%%%%%%%%%%%%%%%%%%%%%%%%%%%%%%%%%
\begin{proof}
Let $\Phi \colon Y_1\times \dots \times Y_d \to X_1\times\dots \times X_d$ establish the isomorphism:
\begin{equation*}
(\cT_1\times \dots \times \cT_d) \circ \Phi = \Phi \circ (\cS_1\times \dots \times \cS_d).
\end{equation*}
Then
\begin{multline*}
\Phi \left(\cS_1\times\dots\times\cS_d\right)^{t_n}\Phi^{-1}=\left(\cT_1\times \dots \times \cT_d\right)^{t_n}\\
 \to \otimes_{i=1}^{d}\left(\left(1-\alpha_i\right)\cdot Id + \alpha_i\cdot \Pi_{X_i,X_i}\right),
\end{multline*}
whence
\begin{equation*}
(\cS_1\times\dots\times\cS_d)^{t_n}\to \Phi^{-1}\circ(  \otimes_{i=1}^{d}((1-\alpha_i)\cdot Id + \alpha_i\cdot \Pi))\circ \Phi.
\end{equation*}
We may assume (passing to a subsequence if necessary) that $S_i^{t_n}\to Q_i$ for some $Q_i\in J(\cS_i)$ for $1\leq i \leq d$. Therefore\footnote{ We identify $\lambda_{Q_1\otimes \dots \otimes Q_d}$ with $\lambda_{Q_1}\otimes \dots \otimes \lambda_{Q_d}$.}
\begin{multline}\label{eq:siedem}
\otimes_{i=1}^{d}Q_i \\
=  \Phi^{-1} \circ \left(\sum_{\varepsilon_i\in\{0,1\},1\leq i\leq d}\otimes_{i=1}^{d}\left( \left(1-\alpha_i\right)^{1-\varepsilon_i}\alpha_i^{\varepsilon_i}\cdot Id^{1-\varepsilon_i} \Pi^{\varepsilon_i}\right)\right) \circ \Phi\\
= \sum_{\varepsilon_i\in\{0,1\},1\leq i\leq d}\Phi^{-1} \circ \left(\otimes_{i=1}^{d}\left( \left(1-\alpha_i\right)^{1-\varepsilon_i}\alpha_i^{\varepsilon_i}\cdot Id^{1-\varepsilon_i} \Pi^{\varepsilon_i}\right)\right) \circ \Phi\\
= \sum_{\varepsilon_i\in\{0,1\},1\leq i\leq d} \prod_{i=1}^{d}(1-\alpha_i)^{1-\varepsilon_i}\alpha_i^{\varepsilon_i}\\
 \cdot\left(  \Phi^{-1} \circ (\otimes_{i=1}^{d}( Id^{1-\varepsilon_i} \Pi^{\varepsilon_i})) \circ \Phi\right)
\end{multline}
where
\begin{equation*} Id^{1-\varepsilon}\Pi^{\varepsilon} = \left\{
\begin{array}{rl} Id & \text{if } \varepsilon=0\\
\Pi & \text{if } \varepsilon=1.
\end{array} \right.
\end{equation*}
Since 
\begin{equation*}
\otimes_{i=1}^{d}Id^{1-\varepsilon_i}\Pi^{\varepsilon_i}\in J^e(\cT_1\times \dots \times \cT_d),
\end{equation*}
by Lemma~\ref{lm:15}
\begin{equation*}
 \Phi^{-1} \circ (\otimes_{i=1}^{d}( Id^{1-\varepsilon_i} \Pi^{\varepsilon_i})) \circ \Phi \in J^e(\cS_1\times \dots \times \cS_d).
\end{equation*}
By taking the projection onto the first two coordinates in~\eqref{eq:siedem}, we obtain
\begin{equation*}
Q_1=\prod_{i=1}^{d}(1-\alpha_i)\cdot Id + \prod_{i=1}^{d}\alpha_i\cdot \Pi +R_1,
\end{equation*}
where $R_1$ is the projection of the remaining Markov operator. Hence
\begin{equation*}
Q_1=\beta_1\cdot Id+ \gamma_1 \cdot \Pi + \delta_1\cdot \Phi_{\rho_1},
\end{equation*}
where $\beta_1\geq (1-\alpha_1)\dots (1-\alpha_d)$, $\gamma_1 \geq \alpha_1\dots \alpha_d$, $\delta_i\geq 0$, $\beta_1+\gamma_1+\delta_1=1$ and $\Phi_{\rho_{1}}$ is such that the probability measure $\rho_1$ is singular with respect to both $\Delta_1$ and $\nu_1\otimes \nu_1$.

In the same way
\begin{equation*}
Q_i=\beta_i \cdot Id+ \gamma_i \cdot \Pi+ \delta_i \cdot \Phi_{\rho_i}
\end{equation*}
for some $\beta_i,\gamma_i, \delta_i\geq 0$ satisfying $\beta_i+\gamma_i+\delta_i=1$ and $\rho_i\perp \nu_i\otimes \nu_i$, $\rho_i \perp \Delta_i$. Hence
\begin{equation}\label{eq:osiem}
Q_1\otimes \dots \otimes Q_d=\sum_{\varepsilon_i^{\beta},\varepsilon_i^{\gamma},\varepsilon_i^{\delta}\in \{0,1\}, \varepsilon_i^{\beta}+\varepsilon_i^{\gamma}+\varepsilon_i^{\delta}=1} \otimes_{i=1}^{d} \beta_i^{\varepsilon_i^{\beta}}\gamma_i^{\varepsilon_i^{\gamma}}\delta_i^{\varepsilon_i^{\delta}}\cdot Id^{\varepsilon_i^{\beta} }\Pi^{\varepsilon_i^{\gamma}}\Phi_{\rho_i}^{\varepsilon_i^{\delta}},
\end{equation}
where
\begin{equation*} Id^{\varepsilon_i^{\beta} }\Pi^{\varepsilon_i^{\gamma}}\Phi_{\rho_i}^{\varepsilon_i^{\delta}} = \left\{
\begin{array}{lll}
 Id & \text{if }\;\; (\varepsilon_i^{\beta},\varepsilon_i^{\gamma},\varepsilon_i^{\delta})=(1,0,0),\\
 \Pi & \text{if }\;\; (\varepsilon_i^{\beta},\varepsilon_i^{\gamma},\varepsilon_i^{\delta})=(0,1,0),\\
\Phi_{\rho_i} & \text{if }\;\; (\varepsilon_i^{\beta},\varepsilon_i^{\gamma},\varepsilon_i^{\delta})=(0,0,1).
\end{array} \right.
\end{equation*}
Both~\eqref{eq:siedem} and~\eqref{eq:osiem} are some decompositions of operator $Q_1\otimes \dots\otimes Q_d$ and~\eqref{eq:siedem} is the ergodic decomposition. In the decomposition~\eqref{eq:siedem} there are $2^d$ summands. Notice that in the decomposition~\eqref{eq:osiem} there are at least $2^d$ operators with non-zero coefficients corresponding to ergodic measures, namely
\begin{equation*}
\otimes_{i=1}^{d}\beta_i^{\varepsilon_i^{\beta}}\gamma_i^{\varepsilon_i^{\gamma}}\delta_i^{\varepsilon_i^{\delta}}\cdot Id^{\varepsilon_i^{\beta}}\Pi^{\varepsilon_i^{\gamma}}\Phi_{\rho_i}^{\varepsilon_i^{\delta}}
\end{equation*}
such that $\varepsilon_i^{\delta}=0$. Each of the remaining operators is of the form
\begin{equation*}
\otimes_{i=1}^{d}\beta_i^{\varepsilon_i^{\beta}}\gamma_i^{\varepsilon_i^{\gamma}}\delta_i^{\varepsilon_i^{\delta}}\cdot Id^{\varepsilon_i^{\beta}}\Pi^{\varepsilon_i^{\gamma}}\Phi_{\rho_i}^{\varepsilon_i^{\delta}},
\end{equation*}
where for some $1\leq i_0\leq d$ we have $\varepsilon_{i_0}^{\delta}=1$. Each of them is a convex combination of the $2^d$ operators corresponding to ergodic measures. This is however impossible since all the measure in the decomposition~\eqref{eq:osiem} are mutually singular, whence none of the operators is a convex combination of the remaining ones. Hence $\delta_i=0$ for $1\leq i \leq d$ and the proof is complete.
\end{proof}
%%%%%%%%%%%%%%%%%%%%%%%%%%%%%%%%%%%%%%%%%%%%%%%%%%%%%%%%%%%%
\begin{lm}\label{lm:19}
Let $a_1,\dots,a_d,b_1,\dots,b_d\in(0,1)$. If the multisets (i.e. sets with elements of a multiplicity possibly greater than one)
\begin{equation*}
\left\{ a_{i_1}\cdot \ldots \cdot a_{i_k}\colon i_1< \dots <i_k,1\leq k\leq d \right\}
\end{equation*}
and
\begin{equation*}
\left\{ b_{i_1}\cdot \ldots \cdot b_{i_k}\colon i_1< \dots <i_k,1\leq k\leq d \right\}
\end{equation*}
are equal then also the multisets $\{a_1,\dots,a_d\}$ and $\{b_1,\dots,b_d\}$ are equal.
\end{lm}
%%%%%%%%%%%%%%%%%%%%%%%%%%%%%%%%%%%%%%%%%%%%%%%%%%%%%%%%%%%%
\begin{proof}
We will show how to determine $a_1,\dots,a_d$ knowing the multiset
\begin{equation*}
M=\left\{ a_{i_1}\cdot \ldots \cdot a_{i_k}\colon i_1< \dots <i_k,1\leq k\leq d \right\}.
\end{equation*}
Notice that the largest number in $M$ is equal to $a_{j_1}$ for some $1\leq j_1\leq d$. Assume that we have found $a_{j_1},\dots,a_{j_s}$ such that 
\begin{equation*}
a_i \leq a_{j_s}\text{ for }i\notin \{j_1,\dots,j_s\}.
\end{equation*}
We will show how to find $a_{j_{s+1}}$ such that
\begin{equation*}
a_i \leq a_{j_{s+1}}\text{ for }i\notin \{j_1,\dots,j_{s+1}\}.
\end{equation*}
Consider the multiset
\begin{equation*}
M'=M\setminus \{a_{i_1}\cdot\ldots\cdot a_{i_k}\colon  i_1<\dots <i_k,\ 1\leq k\leq d,\ i_d\in \{j_1,\dots,j_s\}\}.
\end{equation*}
We claim that the largest number in $M'$ is equal to $a_{j_{s+1}}$ for some $1\leq j_{s+1}\leq d$. Indeed, any other number in $M'$ is a product of numbers between $0$ and $1$ with at least one factor being an element of the set $\{a_i\colon i\notin \{j_1,\dots,j_s\}\}$. By induction the proof is complete.
\end{proof}
%%%%%%%%%%%%%%%%%%%%%%%%%%%%%%%%%%%%%%%%%%%%%%%%%%%%%%%%%%%%
\begin{uw}\label{uw:20}
Fix $t_n\to \infty$. Then for each $\alpha \in (0,1)$ the set $\{f\in L^2\colon f\circ T^{t_n}\to \alpha \int f+(1-\alpha)f \text{ weakly}\}$ is a closed subspace, hence the set of $\alpha$'s for which this subspace is non-empty is an isomorphism invariant.
\end{uw}
%%%%%%%%%%%%%%%%%%%%%%%%%%%%%%%%%%%%%%%%%%%%%%%%%%%%%%%%%%%%
\begin{lm}\label{lm:20}
When the assumptions of Proposition~\ref{pr:gl} are satisfied, $\cS_i$ are $\alpha_{\sigma(i)}$-weakly mixing for some permutation $\sigma\in S(d)$.
\end{lm}
%%%%%%%%%%%%%%%%%%%%%%%%%%%%%%%%%%%%%%%%%%%%%%%%%%%%%%%%%%%%
\begin{proof}
We have
\begin{equation}\label{eq:lm:20:1}
L^2_0(X_1\times \dots \times X_d)= \bigoplus_{1\leq k\leq d} \bigoplus_{i_1<\dots <i_k}L_{i_1,\dots,i_k}^X
\end{equation}
and
\begin{equation}\label{eq:lm:20:2}
L^2_0(Y_1\times \dots \times Y_d)= \bigoplus_{1\leq k\leq d} \bigoplus_{i_1<\dots <i_k}L_{i_1,\dots,i_k}^Y,
\end{equation}
where
\begin{equation*}
L_{i_1,\dots,i_k}^X= \bigotimes_{j=1}^{k} L_0^2(X_{i_j}) \text{ and }L_{i_1,\dots,i_k}^Y= \bigotimes_{j=1}^{k} L_0^2(Y_{i_j}).
\end{equation*}
Notice that $(\cT_1\times \dots \times \cT_d)|_{L^X_{i_1,\dots,i_k}}$ is $1-(1-\alpha_{i_1})\cdot \ldots \cdot (1-\alpha_{i_k})$-weakly mixing and $(\cS_1\times \dots \times \cS_d)|_{L^Y_{i_1,\dots,i_k}}$ is $1-(1-\beta_{i_1})\cdot \ldots \cdot (1-\beta_{i_k})$-weakly mixing. Since $\Phi$ is an isomorphism, for every $f\in L^2_0(X_1\times\dots\times X_d)$ such that $(T_1^{t_n}\times \dots \times T_d^{t_n})f\to \alpha\cdot f$ weakly, we have $(S_1^{t_n}\times \dots \times S_d^{t_n})\Phi f \to \alpha\cdot \Phi f$ weakly. Hence the multisets
\begin{equation*}
\left\{ (1-\alpha_{i_1})\cdot \dots \cdot (1-\alpha_{i_k})\colon i_1< \dots <i_k,1\leq k\leq d \right\}
\end{equation*}
and
\begin{equation*}
\left\{ (1-\beta_{i_1})\cdot \dots \cdot (1-\beta_{i_k})\colon i_1< \dots <i_k,1\leq k\leq d \right\}
\end{equation*}
are equal and by Lemma~\ref{lm:19}, the proof is complete.
\end{proof}
%%%%%%%%%%%%%%%%%%%%%%%%%%%%%%%%%%%%%%%%%%%%%%%%%%%%%%%%%%%%
\begin{proof}[Proof of Proposition~\ref{pr:gl}]
Without loss of generality we may assume that $\alpha_1\geq \dots \geq \alpha_d$. 

Let $1\leq i_1\leq d$ be the largest number such that $\alpha_1=\dots =\alpha_{i_1}$. By Remark~\ref{uw:20}, we have
\begin{equation*}
\Phi(L^2_0(X_i))\subset L_0^2(Y_{\sigma_1(1)})\oplus\dots \oplus L^2_0(Y_{\sigma_1(i_1)})
\end{equation*}
for $1\leq i\leq i_1$, where $\sigma_1 \colon \{1,\dots, i_1\}\to \{1,\dots, d\}$ is an injection such that $\cS_{\sigma_1(i)}$ is $\alpha_1$-weakly mixing for $1\leq i\leq i_1$. Notice that
\begin{equation*}
\Phi \in J^e(\cT_i,\cS_{\sigma_1(1)}\times \dots \times \cS_{\sigma_1(i_1)}).
\end{equation*}
By Remark~\ref{uw:inpart} we obtain that
$
\Phi(L^2_0(X_i))\subset L_0^2(Y_{\sigma(i)})
$
for a unique $\sigma(i)\in \{\sigma_1(1),\dots ,\sigma_1(i_1)\}$. Since $\Phi$ is an isomorphism $$\Phi(L^2_0(X_i))= L_0^2(Y_{\sigma(i)})$$ for $1\leq i\leq i_1$.

If $i_1=d$, the proof is complete. If $i_1<d$ let $1\leq i_2\leq d$ be the largest number such that $\alpha_{i_1+1}=\dots =\alpha_{i_2}$. Notice that the only $L^2_0$-subspaces of $L^2_0(Y_1\times\dots \times Y_d)$ on which we will see $\alpha_{i_2}$-weak mixing and which are not in the $\Phi$-image of $L^2_0(X_1\times\dots \times X_{i_1})$ are $L^2_0(Y_{\sigma_2(i)})$ for $1\leq\sigma_2(i)\leq d$ such that $\cS_{\sigma_2(i)}$ is $\alpha_{i_2}$-weakly mixing. As in the first part of the proof 
$
\Phi(L^2_0(X_i))=L^2_0(Y_{\sigma(i)})
$
for $i_1+1\leq i\leq i_2$, where $\sigma(i)$ is unique and such that $\cS_{\sigma(i)}$ is $\alpha_{i_2}$-weakly mixing.

In finitely many steps we obtain a permutation $\sigma$ of $\{1,\dots ,d\}$ and we complete the proof.
\end{proof}
%%%%%%%%%%%%%%%%%%%%%%%%%%%%%%%%%%%%%%%%%%%%%%%%%%%%%%%%%%%%
As a direct consequence of Proposition~\ref{pr:gl} we obtain the following corollary.

%%%%%%%%%%%%%%%%%%%%%%%%%%%%%%%%%%%%%%%%%%%%%%%%%%%%%%%%%%%%
\begin{wn}\label{co:ce2}
Let $d\geq 1$ and let $\cT_i$ be weakly mixing flows which are $\alpha_i$-weakly mixing for some $\alpha_i\in(0,1)$ and $1\leq i\leq d$. Then the centralizer $C(\cT_1\times \dots \times \cT_d)$ of $\cT_1\times \dots \times \cT_d$ consists of transformations belonging to
\begin{equation*}
C(\cT_{\sigma(1)})\times \dots \times C(\cT_{\sigma(d)}),
\end{equation*}
where the permutation $\sigma\in S(d)$ is such that $\sigma(i)=j$ implies $\cT_i\simeq \cT_j$.
\end{wn}
%%%%%%%%%%%%%%%%%%%%%%%%%%%%%%%%%%%%%%%%%%%%%%%%%%%%%%%%%%%%
It is not clear whether it is possible to obtain a complete counterpart of Proposition~\ref{pr:inftyJP} for $\alpha$-weakly mixing flows. We leave the following question open. 
\begin{equation*}
\begin{array}{l} 
\mbox{Does for $\alpha_i$-weakly mixing $\mathcal{T}_i$ the isomorphism}\\
\mbox{$\mathcal{T}_1\times \mathcal{T}_2\times \dots \simeq \mathcal{S}_1\times \mathcal{S}_2\times \dots$}\\
\mbox{imply that $S_i$ are $\beta_i$-weakly mixing for some $\beta_i$?}
\end{array}
\end{equation*}
Using the methods which proved useful in the finite case, one can prove however some infinite version of Proposition~\ref{pr:gl}. Before we formulate it, let us define the following property.
\begin{df}
We say that the set $\{a_i\colon i\in\N\}$ fulfills condition $WO$ when it is is well-ordered, i.e. when every subset of $\{a_i\colon i\in\N\}$ has the least element.
\end{df}
%%%%%%%%%%%%%%%%%%%%%%%%%%%%%%%%%%%%%%%%%%%%%%%%%%%%%%%%%%%%
\begin{uw}
Note that $\{a_i\colon i\in\N\}$ fulfills condition WO if and only if there are no infinite decreasing subsequences in $\{a_i\colon i\in\N\}$.
\end{uw}
%%%%%%%%%%%%%%%%%%%%%%%%%%%%%%%%%%%%%%%%%%%%%%%%%%%%%%%%%%%%
\begin{pr}
Let $\mathcal{T}_i$ be $\alpha_i$-weakly mixing and $\mathcal{S}_i$ be $\beta_i$-weakly mixing along $t_n\to\infty$. Assume that $\{\alpha_i\colon i\in \N\}$ fulfills condition WO. If $\mathcal{T}_1\times \mathcal{T}_2\times \dots \simeq \mathcal{S}_1\times \mathcal{S}_2\times \dots$ via $\Phi$ then $\Phi$ determines an isomorphism between $\mathcal{T}_i$ for and $\mathcal{S}_{\sigma(i)}$ for some permutation $\sigma \colon \N\to \N$.
\end{pr}
\begin{proof}
We will combine the arguments from the proofs of Lemma~\ref{lm:19}, Lemma~\ref{lm:20} and Proposition~\ref{pr:gl}. Notice first that that as in~\eqref{eq:lm:20:1} and~\eqref{eq:lm:20:2} we have
\begin{equation*}
L^2_0(X_1\times X_2\times \dots)= \bigoplus_{k\geq 1} \bigoplus_{i_1<\dots <i_k}L_{i_1,\dots,i_k}^X
\end{equation*}
and
\begin{equation*}
L^2_0(Y_1\times Y_2\times\dots)= \bigoplus_{k \geq 1} \bigoplus_{i_1<\dots <i_k}L_{i_1,\dots,i_k}^Y,
\end{equation*}
where
\begin{equation*}
L_{i_1,\dots,i_k}^X= \bigotimes_{j=1}^{k} L_0^2(X_{i_j}) \text{ and }L_{i_1,\dots,i_k}^Y= \bigotimes_{j=1}^{k} L_0^2(Y_{i_j}),
\end{equation*}
$(\cT_1\times \cT_2 \times\dots)|_{L^X_{i_1,\dots,i_k}}$ is $1-(1-\alpha_{i_1})\cdot \ldots \cdot (1-\alpha_{i_k})$-weakly mixing and $(\cS_1\times\cS_2\times \dots )|_{L^Y_{i_1,\dots,i_k}}$ is $1-(1-\beta_{i_1})\cdot \ldots \cdot (1-\beta_{i_k})$-weakly mixing. Therefore, with each subspace $L_{i_1,\dots,i_k}^X$ and $L_{i_1,\dots,i_k}^Y$ we can associate numbers $(1-\alpha_{i_1})\cdot \ldots \cdot (1-\alpha_{i_k})$ and $(1-\beta_{i_1})\cdot \ldots \cdot (1-\beta_{i_k})$ respectively.

Let $i_1$ be such that $\alpha_{i_1}\leq \alpha_i$ for all $i\in\N$. Let % For $i\geq 1$ let $\Phi^{-1}(\mathcal{B}_i)$ stand for the factor-$\sigma$-algebra of $\cS_1\times\cS_2\times \dots$ representing $\cT_i$.
$$\N_1:=\{i\in\N\colon \alpha_i=\alpha_{i_1}\}.$$ 
Notice that for $a_i\in (0,1)$ for $1\leq i\leq k$ we have 
\begin{equation}\label{eq:kom}
1-(1-a_1)\cdot\ldots\cdot (1-a_k)>1-(1-a_i)=a_i\text{ for }1\leq i\leq k.
\end{equation}
Therefore the only $L_0^2$-subspaces of $L^2_0(Y_1\times Y_2\times\dots)$ on which we observe $\alpha_{i_1}$-weak mixing are $L^2_0(Y_j)$ such that $\beta_j=\alpha_{i_1}$ (on the other $L^2_0$-subspaces by~\eqref{eq:kom} we observe $\alpha$-weak mixing for larger constants $\alpha$). By Remark~\ref{uw:20}, the set
\begin{equation*}
\N_1':=\{j\in\N\colon \beta_j=\alpha_{i_1}\}
\end{equation*}
is non-empty and for $\beta_j\neq\alpha_{i_1}$ implies $\beta_j>\alpha_{i_1}$ by minimality of $\alpha_{i_1}$. Hence
\begin{equation*}
\Phi\left(L^2_0(X_i)\right)\subset \bigoplus_{j\in\N_1'}L^2_0(Y_j)\text{ for }i\in\N_1.
\end{equation*}
Using e.g. the arguments from the proof of Theorem 2 in~\cite{dJ-L}, we conclude that for all $i \in \N_1$ there exists $\sigma(i)\in\N_1'$ such that $\Phi(L^2_0(X_i))\subset L^2_0(Y_{\sigma(i)})$. Since $\Phi$ is an isomorphism, $$\Phi(L^2_0(X_i))= L^2_0(Y_{\sigma(i)})\text{ for }i\in\N_1$$ and $\sigma \colon \N_1 \to \N_1'$ is bijective (we may reverse the roles of $X_i$'s and $Y_i$'s by considering $\Phi^{-1}$ instead of $\Phi$).

If $\N=\N_1$ the proof is complete. Otherwise, let $i_2\in \N$ be such that $\alpha_{i_2}$ is the smallest number in the set $\{\alpha_i\colon i\notin \N_1\}$. Let 
\begin{equation*}
\N_2:=\{i\in\N\colon \alpha_i=\alpha_{i_2}\}.
\end{equation*}
Notice that $\Phi$ as an isomorphism maps independent $\sigma$-algebras onto independent $\sigma$-algebras. Moreover, the only $L^2_0$-subspaces of $L^2_0(Y_1\times Y_2\times\dots)$ on which we observe $\alpha_{i_2}$-weak mixing and which are not in the $\Phi$-image of $L^2_0(\times_{i\in\N_1}X_i)$ are $L^2_0(Y_j)$ for $j\in \N_2'$ (on other $L^2_0$-subspaces which are not in the $\Phi$-image of $L^2_0(\times_{i\in\N_1}X_i)$ we observe by~\eqref{eq:kom} $\alpha$-weak mixing for larger constants $\alpha$).  Hence, using Remark~\ref{uw:20},
\begin{equation*}
\N_2':=\{j\in\N\colon \beta_j=\alpha_{i_2}\}
\end{equation*}
is non-empty. Notice that for $j\notin \N_1'\cup\N_2'$ we have $\beta_j>\alpha_{i_2}$ (otherwise, by Remark~\ref{uw:20}, $\alpha_{i_2}$ wouldn't be the smallest number in $\{\alpha_i\colon i\notin\N_1\}$). Therefore
\begin{equation*}
\Phi(L^2_0(X_i))\subset\bigoplus_{j\in\N_2'}L^2_0(Y_j) \text{ for }i\in\N_2.
\end{equation*}
As in the first part of this proof, we obtain a bijection $\sigma\colon \N_2\to\N_2'$ such that 
\begin{equation*}
\Phi(L^2_0(X_i))= L^2_0(Y_{\sigma(i)}) \text{ for }i\in\N_2.
\end{equation*}
We complete the proof using transfinite induction (the set $\{\alpha_i\colon i\in\N\}$ is well-ordered).

\end{proof}
%%%%%%%%%%%%%%%%%%%%%%%%%%%%%%%%%%%%%%%%%%%%%%%%%%%%%%%%%%%%
The above proposition implies in particular (as in the finite case) that the centralizer of the infinite product $\cT_1\times\cT_2\times \dots$ is the product of the centralizers of $C(\cT_{1}),\dots, C(\cT_2), \dots$ up to a permutation of coordinates. More precisely, we have the following corollary.
\begin{wn}
Let $\cT_i$ for $i\geq 1$ be $\alpha_i$-weakly mixing for some $\alpha_i\in(0,1)$ such that $\{\alpha_i\colon i\in \N\}$ fulfills condition WO. Then the centralizer $C(\cT_1\times\cT_2\times \dots)$ of $\cT_1\times\cT_2\times \dots$ consists of transformations belonging to
\begin{equation*}
C(\cT_{\sigma(1)})\times C(\cT_{\sigma(2)})\times \dots,
\end{equation*}
where the permutation $\sigma\colon \N\to\N$ is such that $\sigma(i)=j$ implies $\cT_i\simeq \cT_j$.
\end{wn}

%%%%%%%%%%%%%%%%%%%%%%%%%%%%%%%%%%%%%%%%%%%%%%%%%%%%%%%%%%%%
%\newpage
\section*{Acknowledgements}
I would like to thank Professor V.V. Ryzhikov for stating the problem and Professor M. Lema{\'n}czyk for fruitful discussions on the subject. I would also like to thank the referee for comments, suggestions and questions.

%\newpage
\small{
\bibliography{bibryztro}}

\begin{thebibliography}{10}

\bibitem{dJ-L}
A.~del Junco and M.~Lema{\'n}czyk.
\newblock Generic spectral properties of measure-preserving maps and
  applications.
\newblock {\em Proc. Amer. Math. Soc.}, 115(3):725--736, 1992.

\bibitem{dJ-R}
A.~del Junco and D.~Rudolph.
\newblock On ergodic actions whose self-joinings are graphs.
\newblock {\em Ergodic Theory Dynam. Systems}, 7(4):531--557, 1987.

\bibitem{FLold}
K.~Fr\k{a}czek and M.~Lema{\'n}czyk.
\newblock On symmetric logarithm and some old examples in smooth ergodic
  theory.
\newblock {\em Fund. Math.}, 180(3):241--255, 2003.

\bibitem{FL?}
K.~Fr\k{a}czek and M.~Lema{\'n}czyk.
\newblock On disjointness properties of some smooth flows.
\newblock {\em Fund. Math.}, 185(2):117--142, 2005.

\bibitem{Furstenberg67}
H.~Furstenberg.
\newblock Disjointness in ergodic theory, minimal sets, and a problem in
  {D}iophantine approximation.
\newblock {\em Math. Systems Theory}, 1:1--49, 1967.

\bibitem{joi0}
E.~Glasner.
\newblock {\em Ergodic theory via joinings}, volume 101 of {\em Mathematical
  Surveys and Monographs}.
\newblock American Mathematical Society, Providence, RI, 2003.

\bibitem{Ka01}
A.~Katok.
\newblock {\em Combinatorial constructions in ergodic theory and dynamics},
  volume~30 of {\em University Lecture Series}.
\newblock American Mathematical Society, Providence, RI, 2003.

\bibitem{Koc76}
A.~V. Kochergin.
\newblock Nondegenerate saddles, and the absence of mixing.
\newblock {\em Mat. Zametki}, 19(3):453--468, 1976.

\bibitem{lemanczyk+parreau}
M.~Lema{\'n}czyk and F.~Parreau.
\newblock Special flows over irrational rotations with simple convolution
  property.
\newblock {\em Preprint}, 2009.

\bibitem{lpr}
M.~Lema{\'n}czyk, F.~Parreau, and E.~Roy.
\newblock Joining primeness and disjointness from infinitely divisible systems.
\newblock {\em Proc. Amer. Math. Soc.}, 139(1):185--199, 2011.

\bibitem{joi2}
M.~Lema{\'n}czyk, F.~Parreau, and J.-P. Thouvenot.
\newblock Gaussian automorphisms whose ergodic self-joinings are {G}aussian.
\newblock {\em Fund. Math.}, 164(3):253--293, 2000.

\bibitem{MR981173}
D.~S. Ornstein.
\newblock Ergodic theory, randomness, and ``chaos''.
\newblock {\em Science}, 243(4888):182--187, 1989.

\bibitem{joi1}
V.~V. Ryzhikov.
\newblock Joinings, wreath products, factors and mixing properties of dynamical
  systems.
\newblock {\em Izv. Ross. Akad. Nauk Ser. Mat.}, 57(1):102--128, 1993.

\bibitem{ryz}
V.~V. Ryzhikov.
\newblock Genericity of the isomorphism of measure-preserving transformations
  under isomorphism of their {C}artesian powers.
\newblock {\em Mat. Zametki}, 59(4):630--632, 1996.

\bibitem{Ry-Th}
V.~V. Ryzhikov and J.-P. Thouvenot.
\newblock Disjointness, divisibility, and quasi-simplicity of
  measure-preserving actions.
\newblock {\em Funktsional. Anal. i Prilozhen.}, 40(3):85--89, 2006.

\bibitem{ryz-tro}
V.~V. Ryzhikov and A.~E. Troitskaya.
\newblock The tensor root of an isomorphism and weak limits of transformations.
\newblock {\em Mat. Zametki}, 80(4):596--600, 2006.

\bibitem{St}
A.~M. Stepin.
\newblock Spectral properties of generic dynamical systems.
\newblock {\em Izv. Akad. Nauk SSSR Ser. Mat.}, 50(4):801--834, 879, 1986.

\bibitem{Tro}
A.~E. Troitskaya.
\newblock On the isomorphism of measure-preserving {$\Bbb Z^2$}-actions that
  have isomorphic {C}artesian powers.
\newblock {\em Fundam. Prikl. Mat.}, 13(8):193--212, 2007.

\end{thebibliography}
\end{document}